\documentclass[11pt,reqno]{article} 
\usepackage[latin1]{inputenc}   
\usepackage{amsfonts,amsmath,layout}
\usepackage{mathrsfs}  
\usepackage{amssymb,mathabx,amsfonts,amsthm,amscd,stmaryrd,dsfont,esint,upgreek,constants,todonotes}
\DeclareFontFamily{OT1}{pzc}{}
\DeclareFontShape{OT1}{pzc}{m}{it}{<-> s * [1.10] pzcmi7t}{}
\DeclareMathAlphabet{\mathpzc}{OT1}{pzc}{m}{it}

\usepackage{color} 

\definecolor{Red}{cmyk}{0,1,1,0.2}
\usepackage{hyperref}

\newcommand{\N}{\mathbb N}

\newcommand{\Z}{\mathbb Z}

\newcommand{\R}{\mathbb R}

\def\R{\mathbb R}
\def\N{\mathbb N}
\def\Z{\mathbb Z}

\def\E{\mathbb E}
\def\P{\mathbb P}

\def\ep{\epsilon}

\newcommand{\be}{\begin{equation}}
\newcommand{\ee}{\end{equation}}
\def\1{{\bf 1}}

\newtheorem{Theorem}{Theorem}[section]
\newtheorem{Definition}[Theorem]{Definition}
\newtheorem{Proposition}[Theorem]{Proposition}
\newtheorem{Lemma}[Theorem]{Lemma}
\newtheorem{Corollary}[Theorem]{Corollary}

\begin{document}

\title{From heterogeneous microscopic traffic flow models to macroscopic models}
\author{\renewcommand{\thefootnote}{\arabic{footnote}}
  P. Cardaliaguet\footnotemark[1], N. Forcadel\footnotemark[2]}
\footnotetext[1]{CEREMADE, UMR CNRS 7534, Universit\'e Paris-Dauphine PSL,
Place de Lattre de Tassigny, 75775 Paris Cedex 16, France}
\footnotetext[2]{Normandie Univ, INSA de Rouen Normandie, LMI (EA 3226 - FR CNRS 3335), 76000 Rouen, France, 685 Avenue de l'Universit\'e, 76801 St Etienne du Rouvray cedex.
France}

\maketitle

\begin{abstract}
The goal of this paper is to derive rigorously macroscopic traffic flow models from microscopic models. More precisely, for the microscopic models, we consider follow-the-leader type models with different types of drivers and vehicles which are distributed randomly on the road. After a rescaling, we show that the cumulative distribution function converge to the solution of a macroscopic model. We also make the link between this macroscopic model and the so-called LWR model.
\end{abstract}



\section{Introduction}
The modelling and the simulation of traffic flow is a challenging task in particular in order to design infrastructure. In particular, it can allow us to understand how the traffic will react to a change in the infrastructure of the road (if it is interesting to place a traffic light, if a bridge would help the traffic flow, how would a moderator affect the traffic...). Indeed, there are some examples in which the construction of a new infrastructure did not improve the traffic. For example, in Stuttgart, Germany, after investments into the road network in 1969, the traffic situation did not improve until a section of newly built road was closed for traffic again (see \cite{K69}). This is known as the Braess' paradox. During the last years, a lot of work has been done concerning the modelling of traffic flows problems. 

The goal of this paper is to obtain rigorously macroscopic trafic flow models by rescaling microscopic models. From a modelling point of view, microscopic models describe the dynamics of each vehicles individually. The most famous microscopic models are those of follow-the-leader type, also known as car-following models. In such a model, the dynamics of each vehicle depends on the vehicles in front, so that, in a cascade, the whole traffic flow can be determined by the dynamics of the very first vehicle (the leader). Two very popular models are the Bando model \cite{ModelBando} and the Newell model \cite{newell}, which describe respectively the acceleration and the velocity of each vehicle as a function (called optimal velocity function) of the inter-distance with the vehicle in front of it. The main advantage of these methods is that one can easily distinguish each vehicle and then associate different attributes (like maximal velocity, maximal acceleration...) to each vehicle. It is also possible to describe microscopic phenomena like red lights, slowdown or change of the maximal velocity. The main drawback is for numerical simulations where we have to treat a large number of data, which can be very expensive for example if we want to simulate the traffic at the scale of a town. 

On the contrary, macroscopic models consist in describing the collective behaviour of the particles for example by giving an evolution law on the density of vehicles. Concerning the modelling of the traffic flow, the oldest macroscopic model is the LWR model (Lighthill, Whitham \cite{LW}, Richards \cite{richards}, see also the book \cite{piccoli} for a good introduction to this models), which dates back to 1955 and is inspired by the laws of fluid dynamics. More recently, some macroscopic models propose to describe the flow of vehicles in terms of the averaged spacing between the vehicles (in some sense, the inverse of the density, see the works of Leclercq, Laval and Chevallier \cite{leclercq}). The main advantage of these macroscopic models is that it is possible to make numerical simulations on large portion of road. On the other hand, it is more complicated to describe microscopic phenomena or attributes. Generally speaking, microscopic models are considered more justifiable because the behaviour of every single particle can be described with high precision and it is immediately clear which kind of interactions are considered. On the contrary, macroscopic models are based on assumptions that are hardly correct or at least verifiable. As a consequence, it is often desirable establishing a connection between microscopic and macroscopic models so to justify and validate the latter on the basis of the verifiable modelling assumptions of the former.

Connections between microscopic and macroscopic fluid-dynamics traffic flow models are already well understood in certain cases of vehicles moving on a single road. The goal of this paper is to generalise these results to the case of different types of drivers which are distributed randomly on the road. The existing results mainly consider the case of one type of vehicles (see for example \cite{rigorousLWR} or \cite{GoatinRossi} for the case of non-local intercations) or several types of vehicles/drivers but distributed in a periodic way on the road (see \cite{homogenisationOVF}). To the best of our knowledge, the only result in the random case is a formal derivation obtain in the paper of Chiabaut, Leclercq and Buisson \cite{CLB09}.

\subsection{Description of the model and assumptions}
In this paper, we consider a stochastic first order follow-the-leader model. More precisely, we assume that each car has a type $z$ and that the optimal velocity $V_z(p)$ of a car, which expresses the velocity of the car in function of the distance $p$ of this car to the car in front of it, depends on this type $z$.  The set of types is a compact metric space $({\mathcal Z}, d_{\mathcal Z})$. A typical (and realistic) example is when ${\mathcal Z}$ is a finite set. 
We suppose that the cars are labelled by an index $i$, with $i\in \Z$, the position of the cars being an increasing function of $i$. Our main structure assumption is that the type of car $i$ is a random variable $Z_i$ with values in ${\mathcal Z}$ and that the $(Z_i)$ are i.i.d. random variables on a probability space $(\Omega, {\mathcal F}, \P)$. The types $(Z_i)$ do not evolve in time and are  fixed throughout the paper. {\it Without loss of generality we also suppose throughout the paper that the law of $Z_0$ (and therefore of all the $Z_i$) has full support,} since we can always restrict the compact metric space $({\mathcal Z}, d_{\mathcal Z})$ to the support of the law of $Z_0$.

On the optimal velocity map $V:{\mathcal Z}\times \R_+\to \R$, we assume the following:
\begin{itemize}
\item[$(H_1)$] The map $(z,p)\to V_z(p)$ is uniformly continuous on ${\mathcal Z}\times \R_+$ and $p\to V_z(p)$ is Lipschitz continuous, uniformly with respect to $z\in {\mathcal Z}$;
\item[$(H_2)$] For any $z\in {\mathcal Z}$, there exists $h_0^z>0$ (depending in a measurable way on $z$) such that $V_z(p)=0$ for all $p\in[0,h_0^z]$;
\item[$(H_3)$] For any $z\in {\mathcal Z}$,  $p\to V_z(p)$ is increasing in $[h_0^z,+\infty)$;
\item[$(H_4)$] There exists $V_{\max}>0$ and, for any $z\in {\mathcal Z}$, there exists $V^z_{\max}\leq V_{\max}$, such that $\lim_{p\to +\infty}V_z(p)=V^z_{\max}$.\\
Under assumptions (H1)---(H4), the map $V_z:  [h_0^z,+\infty)\to [0, V^z_{\max})$ is increasing and continuous for any $z\in {\mathcal Z}$ and we denote by $V^{-1}_z$ its inverse. For simplicity of notation, we set
$\bar h_0:=\E[h_0^{Z_0}]$ and by $\underline 
{V_{\max}}:=\inf_{z\in {\mathcal Z}} V_{\max}^z$. We need a last assumption, on the statistical distribution of $V^{Z_0}_{\max}$: 
\item[$(H_5)$] $\displaystyle \lim_{\theta\to \underline {V_{\max}}^-} \E\left[ V^{-1}_{Z_0} (\theta)\right] = +\infty.$   
\end{itemize}
The last condition (H5) is merely technical. It is clearly satisfied if ${\mathcal Z}$ is finite because in this case $\P[V^{Z_0}_{\max}= \underline {V_{\max}}]>0$. It roughly says that, statistically, $V^{Z_0}_{\max}$ is close to $\underline {V_{\max}}$ with a sufficiently large probability. The schematic representation of the optimal velocity functions is given in Fig. \ref{OVF}. 
\begin{figure}[!ht]
  \centering
  \includegraphics[width=120mm]{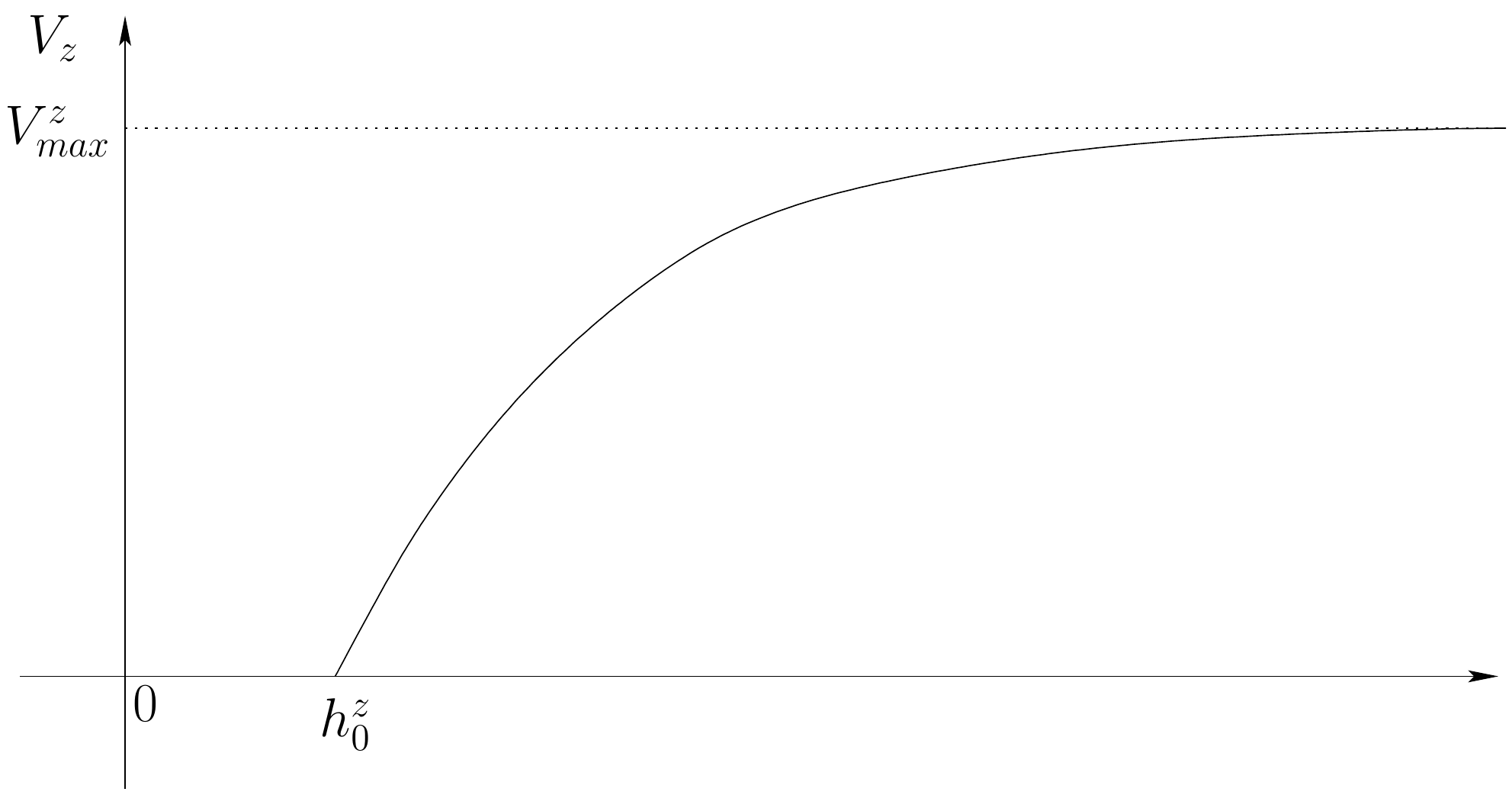}
  \caption{Schematic representation of the optimal velocity functions.}
  \label{OVF}
\end{figure}

With these conditions in mind, we consider the (random) follow-the-leader model :
\begin{equation}\label{eq:model}
\frac{d}{dt} U_i(t)= V_{Z_i} (U_{i+1}(t)-U_i(t)), \; t\geq 0,  \forall i\geq 0. 
\end{equation}
In this model, $U_i$ denotes the position of car $i$ and $\frac{d}{dt} U_i$ its velocity. 

\subsection{Main results}
We now want to rescale the microscopic model. For $\epsilon>0$, we consider an initial condition $(U^{ \ep,0}_i)$ such that there exists a Lipschitz continuous function $u_0:\R\to \R$ with
\begin{equation}\label{eq:initcond}
\lim_{\ep\to 0, \ \ep i \to x} \ep U^{ \ep,0}_i = u_0(x), 
\end{equation}
locally uniformly with respect to $x$. Let $(U^\ep_i)$ be the solution of \eqref{eq:model} with initial condition $(U^{\ep,0}_i)$. 

\begin{Theorem}[Convergence result]\label{th:CV}
Under assumptions $(H1)-(H5)$,  the limit 
$$
u(x,t):= \lim_{\ep \to 0,\  \ep(i,s)\to (x,t)} \ep U^\epsilon_i(s)
$$
exists a.s., locally uniformly in $(x,t)$, and  $u$ is the unique (deterministic) viscosity solution of 
\begin{equation}\label{eq:macro}\left\{\begin{array}{ll}
\partial_t u=\bar F(\partial_x u) & {\rm in}\; \R\times ]0,+\infty[\\
u(x,0)=u_0(x)& {\rm in}\; \R
\end{array}
\right.
\end{equation}
where the effective velocity $\bar F:[0,+\infty)\to [0,\bar V_{\max})$ is the continuous and increasing map defined by $\bar F(p)=0$ if $p\le \bar h_0$ and $\E[V^{-1}_{Z_0}(\bar F(p))]=p$ if $p>\bar h_0$. 
\end{Theorem}

\subsection{Link with the LWR model}
In this subsection, we explicit the link between the macroscopic model \eqref{eq:macro} and the classical Lighthill-Whitham-Richards (LWR) model (see \cite{LW,richards}).
Let $U^\ep$ be the solution of \eqref{eq:model} with an initial condition $U^{ \ep,0}$ such that \eqref{eq:initcond} holds. We consider the (rescaled) empirical density of cars: 
$$
\rho^\ep(t) = \ep  \sum_{i\in \Z} \delta_{\ep U^\ep_{i} (t/\ep)}, \qquad t\geq 0. 
$$

\begin{Corollary} As $\ep\to 0$, $\rho^\ep(t)$ converges, a.s., in distribution and locally uniformly in time, to 
$$
\rho(t):= u(\cdot, t)\sharp dx,
$$
 where $u$ is the solution of \eqref{eq:macro} and the notation $u(\cdot, t)\sharp dx$ denotes the push-forward measure of $dx$ by $u$. If, in addition, there exists $C>0$ such that
\begin{equation}\label{hyp:add}
C^{-1} \leq \partial_x u_0(x) \leq C,
\end{equation}
then $\rho$ has an absolutely continuous density which is locally bounded and is the entropy solution of the LWR model 
\begin{equation}\label{eq:LWR}
\partial_t \rho+ \partial_x(\rho \bar v(\rho))= 0\qquad {\rm in}\; \R\times \R_+,
\end{equation}
with initial condition $u_0(\cdot)\sharp dx$ and where $\bar v(\rho)= \bar F(1/\rho)$. 
\end{Corollary}

\begin{proof} Let $\varphi\in {\mathcal C}^0_c(\R)$. Then, for any $t'\geq 0$,  
\begin{align*}
\int_{\R} \varphi(x) \rho^\ep(dx,t') & = \ep \sum_{i\in \Z} \varphi(\ep U^\ep_{i} (t'/\ep))  = \int_{\R} \varphi(\ep U^\ep_{[x/\ep]} (t'/\ep))  dx. 
\end{align*}
As $\ep([x/\ep], t'/\ep)\to (x, t)$ as $\ep\to0$ and $t'\to t$, Theorem \ref{th:CV} implies: 
\begin{align*}
\lim_{\ep \to 0, \; t'\to t} \int_{\R} \varphi(x) \rho^\ep(dx,t) & =  \int_{\R} \varphi(u(x,t))  dx = \int_\R \varphi(x) d(u(\cdot,t)\sharp dx)=\int_\R\varphi(x)\rho(dx,t). 
\end{align*}
This proves that $\rho^\ep(t)$ converges locally uniformly in time and in the sense of measures to $\rho(t):= u(\cdot,t)\sharp dx$. 

Let us now assume that  \eqref{hyp:add} holds. From standard arguments (as $\bar F$ does not depend on space), one easily checks that the bounds \eqref{hyp:add} are preserved in time: 
$$
C^{-1} \leq \partial_x u(x,t) \leq C.
$$
In particular $\rho(t)$ is absolutely continuous with a density satisfying the same estimate: 
$$
C^{-1} \leq \rho(x,t) \leq C.
$$
In addition, one easily checks that  $w(x,t):= u^{-1}(x,t)$ (where the inverse is taken in space) is a viscosity solution to 
$$
\partial_t w +\partial_x w \bar F(1/\partial_xw)=0,
$$
which in turn implies that $\rho(x,t)= \partial_x w(x,t)$ is an entropy solution to \eqref{eq:LWR}. 
\end{proof}

\section{Some properties of the microscopic and macroscopic models}
\subsection{The microscopic model}
The following lemma is a direct consequence of the fact that $V_z$ is non-negative joint to assumption $(H_4)$. 
\begin{Lemma}
Let $U_i$ be a solution of \eqref{eq:model}. Then, for all $t\ge 0$,
$$0\le U_i(t)-U_i(0)\le V_{\rm max}t.$$
\end{Lemma}
\begin{Proposition}\label{prop.compar}
Let $U_i$ and $\tilde U_i$ be two solutions of \eqref{eq:model} such that there exists $i_0\ge 0$ such that 
$$U_i(0)\le \tilde U_i(0)\quad \forall i\ge i_0.$$
Then
$$U_i(t)\le \tilde U_i(t)\quad \forall t\ge 0 \textrm{ and }  i\ge i_0.$$
\end{Proposition}

\begin{proof}
Let $T\ge 0$. We want to prove that 
$$M=\sup_{t\in[0,T]}\max_{i\ge i_0}(U_i(t)-\tilde U_i(t))\le 0.$$
By contradiction, assume that $M>0$ and consider for $\varepsilon, \eta,\beta>0$
$$M_{\varepsilon}=\sup_{t,s\in[0,T]}\max_{i\ge i_0}\left\{U_i(t)-\tilde U_i(s)-\frac{(t-s)^2}{2\varepsilon}-\frac \eta{T-t}-\beta (i-i_0)^2\right\}>0$$
for $\eta$ and $\beta$ small enough. We assume that the maximum is reached at some point $(t,s)\in[0,T]$ and for an index $i\ge i_0$. We first prove that $t>0$ and $s>0$. Indeed, if $t=0$, then we get
$$\frac {s^2}{2\varepsilon}+\frac {\eta}T<U_i(0)-\tilde U_i(s)\le U_i(0)-\tilde U_i(0)\le0$$
which is a contradiction.
In the same way, if $s=0$, then
$$\frac{t^2}{2\varepsilon}+\frac \eta {T}\le U_i(t)-\tilde U_i(0)\le V_{\rm max}t$$
which is a contradiction for $\varepsilon$ small enough.
We then deduce that $t,s>0$. This implies that
$$\frac \eta{T^2}\le V_{Z_i}(U_{i+1}(t)-U_i(t))-V_{Z_i}(\tilde U_{i+1}(s)-\tilde U_i(s)).$$
Using that
$$U_{i+1}(t)-\tilde U_{i+1}(s)-\beta(i+1-i_0)^2\le U_{i}(t)-\tilde U_{i}(s)-\beta(i-i_0)^2,$$
and the fact that $V_Z$ are uniformly Lipschitz continuous, we deduce that
$$\frac \eta{T^2}\le C\beta(1+2(i-i_0)).$$
Sending $\beta\to 0$ (and using the classical fact that $\beta(i-i_0)\to0$), we get a contradiction.
\end{proof}
\subsection{The macroscopic model}
In this subsection, we recall the definition of viscosity solution for the macroscopic model \eqref{eq:macro} and give a comparison principle.

\begin{Definition}[Definition of viscosity solutions for \eqref{eq:macro}]
\label{def:macro}
An upper semi-continuous (resp. lower semi-continuous) function $u: \mathbb R \times [0,+\infty) \rightarrow \mathbb{R}$ is a viscosity sub-solution (resp. super-solution) of \eqref{eq:macro} if $u(x,0)\leq {u}_0(x)$ (resp. $u(x,0)\geq {u}_0(x)$) and for all $(x,t)\in  \mathbb R \times [0,+\infty)$ and for all $\varphi\in C^1 ( \mathbb R \times [0,+\infty)$ such that $u-\varphi$ reaches a local maximum (resp. minimum) in $(x,t)$, we have
\begin{equation}
\partial _t\varphi(x,t)\le \bar F(\partial _x\varphi(x,t))\quad \left({\rm resp.}\; \partial _t\varphi(x,t)\ge \bar F(\partial _x\varphi(x,t))\right)
\end{equation}
We say that $u$ is a viscosity solution of \eqref{eq:macro} if $u^*$ and $u_*$ are respectively a sub-solution and a super-solution of \eqref{eq:macro}. 
\end{Definition}

\begin{Proposition}[Comparison principle for \eqref{eq:macro}]\label{pro:PCmacro}
Let $u$ and $v$ be respectively a sub and a super-solution of \eqref{eq:macro} and assume that $u_0$ is Lipschitz continuous. Then
$$u\le v\quad{ in}\; \mathbb R \times [0,+\infty).$$
\end{Proposition}

\section{Construction of the effective velocity}
Recall that $\bar h_0:=\E[h_0^{Z_0}]$. Given $p>\bar h_0$, in order to construct the effective velocity $\bar F(p)$,  we consider the solution to 
\begin{equation}\label{eq:lineardata}
\frac{d}{dt} U_i(t)= V_{Z_i} (U_{i+1}(t)-U_i(t)), \; t\geq 0, \qquad U_i(0)= pi \qquad \forall i\geq 0. 
\end{equation}
We begin to prove that, as $t\to \infty$, $U_i(t)/t$ is bounded by $\underline 
{V_{\max}}:=  \inf_{z\in {\mathcal Z}} V_{\max}^z$.
\begin{Proposition} Let $(U_i)$ be the solution of \eqref{eq:lineardata}. Then, for every $i\ge 0$, we have,
$$ \limsup_{t\to\infty} \;\frac {U_i(t)}{t}\le  \inf_{z\in {\mathcal Z}} V_{\max}^z=:\underline 
{V_{\max}}.
$$
\end{Proposition}

\begin{proof}
Let $\varepsilon>0$ and fix $i\ge 0$. Since the $Z_j$ are i.i.d. and have a law which has a full support in ${\mathcal Z}$, there exists $i_0\ge i$ such that
$$V_{\max}^{Z_{i_0}}\le \underline{V_{\max}} +\varepsilon.$$
Moreover, we have
$$U_{i_0}(t)\le U_{i_0}(0)+V_{\max}^{Z_{i_0}}\cdot t.$$
Since the $U_j$ are increasing in $j$, we deduce that
$$\frac {U_i(t)}{t}\le \frac{U_{i_0}(t)}t\le \frac {U_{i_0}(0)}t+V_{\max}^{Z_{i_0}}.$$
This implies that
$$\limsup_{t\to +\infty} \frac {U_i(t)}t\le V_{\max}^{Z_{i_0}}\le \underline{V_{\max}}+\varepsilon.$$
We finally get the result by taking $\varepsilon\to 0$.
\end{proof}

We now construct the limit of $U_i(t)/t$.
\begin{Proposition}\label{prop.limlin}  There exists $\Omega_0$ with $\mathbb P(\Omega_0)=1$  such that for every $p\ge 0$, $i\in \N$ and $\omega\in \Omega_0$
$$\lim_{t\to+\infty} \frac{ U_i(t)}{t} = \bar F(p)\qquad \forall i\geq 0, 
$$
where   $(U_i)$ is the solution of \eqref{eq:lineardata} and where the continuous and non-decreasing map $\bar F:\R_+\to \R_+$ is defined by $\bar F(p)=0$ if $p\le \bar h_0$ and by the relation $\E[V^{-1}_{Z_0}(\bar F(p))]=p$ if $p>\bar h_0$. 
\end{Proposition}

\begin{proof} Let us first check that the map $\bar F$ is well defined and continuous. For this we first recall that  $V_z^{-1}:  [0, V^z_{\max})\to [h_0^z,+\infty)$ is increasing and continuous. We claim that, for any $\theta\in [0, \underline {V_{\max}})$,  $V_z^{-1}(\theta)$ is bounded with respect to $z\in {\mathcal Z}$. Indeed, assume that, contrary to our claim, there exists $z_n\in {\mathcal Z}$ such that $p_n:=V_{z_n}^{-1}(\theta)\to+\infty$. As ${\mathcal Z}$ is compact, there exists $z$ a limit of a subsequence of the $(z_n)$, still denoted $z_n$. Then, by uniform continuity of the map $V$,  
$$
\theta= \lim_n V_{z_n}(p_n) = V^z_{\max} \geq \underline {V_{\max}} >\theta, 
$$
which is a contradiction. Therefore the map $\theta\to \E\left[ V_{Z_0}^{-1}(\theta)\right]$ is well-defined, increasing  and continuous on $ [0, \underline {V_{\max}})$. Moreover, by monotone convergence, $\E[V^{-1}_{Z_0}(\theta)]\to \E[h_0^{Z_0}]=: \bar h_0$ as $\theta\to 0^+$ while, by assumption (H5), we have 
$$
\lim_{\theta\to \underline{V_{\max}}}\E[V^{-1}_{Z_0}(\theta)]=+\infty.
$$
So we can define, for $p\in(\bar h_0,+\infty)$,   $\bar F(p)$ as the unique real number in $(0,\underline{V_{\max}})$ such that 
$$
 \E\left[ V_{Z_0}^{-1}(\bar F(p)) \right]=p.
$$
We extend $\bar F$ by setting $\bar F(p)= 0$ for $p\in [0,  \bar h_0]$. Then $\bar F$ is increasing (in $[\bar h_0,+\infty)$), continuous and satisfies 
$$
\lim_{p \to \bar h_0^+} \bar F(p)= 0, \qquad \lim_{p \to +\infty} \bar F(p)= \underline {V_{\max}}.
$$

\medskip

In order to study the behavior of the solution to \eqref{eq:lineardata}, we introduce next the correctors of the problem. Given  $\theta\in (0,\underline {V_{\max}})$, we consider the random sequence 
$(c_i^\theta )$ defined by $c_0^\theta=0$ and $c_{i+1}^\theta = c_i^\theta+V_{Z_i}^{-1}(\theta)$. In other words, 
\be\label{ohebfdnl}
V_{Z_i}(c_{i+1}^\theta-c_i^\theta) = \theta \qquad \forall i\geq 0.
\ee
The reals $c_i^\theta$ represent the optimal position of the vehicle $i$ in the sense that all the vehicles drive at velocity $\theta$. In this sense, the sequence $(c^\theta_i)$ is related to the so-called hull functions (see the pionnering work of Aubry \cite{A1,A2}, and Aubry, Le Daeron \cite{ALD} as well as \cite{overdampedFKM,homogenisationOVF}).
We note that, by the law of large numbers,  there exists   $\Omega_0$ with $\mathbb P(\Omega_0)=1$  such that for every  $\omega\in \Omega_0$, we have
\be\label{piazrmhsndm}
\frac{c_n^\theta}{n} =  \frac{1}{n} \sum_{i=0}^{n-1} V_{Z_i}^{-1}(\theta) \to \E\left[ V_{Z_0}^{-1}(\theta) \right]\qquad \mbox{\rm as $n\to +\infty$}. 
\ee

\medskip

We now study the limit of $\frac {U_i(t)}t$ as $t\to +\infty$. Let us fix $p>h_0$ and $\ep>0$ small. We first define 
$$
\tilde U_i(t)= c_i^{p+\ep}+t \bar F(p+\ep),
$$
(where by a ``small" abuse of notation, $c_i^{p+\ep}:= c_i^{\bar F(p+\ep)}$, i.e., $\theta =\bar F(p+\ep)$ ). Note that, by \eqref{ohebfdnl}, $\tilde U_i$ solves 
$$
V_{Z_i}(\tilde U_{i+1}(t)-\tilde U_i(t))= V_{Z_i}(c_{i+1}^{p+\ep}-c_i^{p+\ep}) =\bar F(p+\ep) = \frac{d}{dt} \tilde U_i(t),\qquad \forall i\geq 0.
$$
 Moreover, by \eqref{piazrmhsndm}, we have a.s.
$$
\lim_n \frac{c_n^{p+\ep}}{n} =  \E\left[ V_{Z_0}^{-1}(\bar F(p+\ep))) \right] =p+\ep.
$$
Thus (since $U_i(0)=pi$), there exists a random integer $i_0$ such that 
$$
U_i(0)\leq \tilde U_i(0) \qquad \forall i\geq i_0. 
$$
By comparison (Proposition \ref{prop.compar}) we obtain 
$$
U_i(t)\leq \tilde U_i(t) = \tilde U_i(0)+t\bar F(p+\ep) \qquad \forall i\geq i_0, \; t\geq 0. 
$$
Hence, for any $i\geq i_0$, 
$$
\limsup_{t\to+\infty} \frac{U_i(t)}{t}\leq \bar F(p+\ep).
$$
This actually holds for any $i$ (since $i\to U_i(t)$ is increasing) and for any $\ep>0$. So 
$$
\limsup_{t\to+\infty} \frac{U_i(t)}{t}\leq \bar F(p)\qquad \forall i\geq 0.
$$

Next we want to prove the bound below, which is slightly more difficult. As above we set 
$$
\tilde U_i(t)= c_i^{p-\ep}+t \bar F(p-\ep),
$$
(where by again a ``small" abuse of notation, $c_i^{p-\ep}:= c_i^{\bar F(p-\ep)}$). Note that $\tilde U_i$ solves 
$$
V_{Z_i}(\tilde U_{i+1}(t)-\tilde U_i(t))= V_{Z_i}(c_{i+1}^{p-\ep}-c_i^{p-\ep}) =\bar F(p-\ep) = \frac{d}{dt} \tilde U_i(t),\qquad \forall i\geq 0
$$
and, by \eqref{piazrmhsndm}, we have a.s.
$$
\lim_n \frac{c_n^{p-\ep}}{n} =  \E\left[ V_{Z_0}^{-1}(\bar F(p-\ep))) \right] =p-\ep.
$$
Thus there exists a random integer $i_0$ such that 
$$
U_i(0)\geq \tilde U_i(0) \qquad \forall i\geq i_0. 
$$
By comparison (Proposition \ref{prop.compar}) we obtain 
$$
U_i(t)\geq  \tilde U_i(t) = \tilde U_i(0)+t\bar F(p-\ep) \qquad \forall i\geq i_0, \; t\geq 0. 
$$
Hence, for any $i\geq i_0$, 
$$
\liminf_{t\to+\infty} \frac{U_i(t)}{t}\geq \bar F(p-\ep).
$$
It remains to check that this inequality still holds for $i=0$ (and thus for all $i\geq 0$). 

For this we set 
$$
W_i(t)= U_i(t)-\tilde U_i(t),
$$
and note that 
\begin{align*}
\frac{d}{dt} W_i(t) & = V_{Z_i}(U_{i+1}(t)-U_i(t))-V_{Z_i}(\tilde U_{i+1}(t)-\tilde U_i(t)) \\
& = V_{Z_i}'(\sigma_i(t))\ (W_{i+1}(t)-W_i(t))
\end{align*}
for some $\sigma_i(t)\geq h_0$ that can be chosen measurable with respect to $t$ and $\omega$. We integrate this equality to get 
$$
W_i(t) = W_i(0) \exp\Bigl\{-\int_0^t V_{Z_i}'(\sigma_i(\tau))d\tau\Bigr\} + \int_0^t  V_{Z_i}'(\sigma_i(s))\exp\left\{-\int_s^t V_{Z_i}'(\sigma_i(\tau))d\tau\right\} W_{i+1}(s)ds.
$$
Now we consider the indices $i\in \{0, \dots i_0\}$ and set $C=\sup_{j\in \{0, \dots, i_0\}} |W_i(0)|$ (this is a random variable). 
As $V_k'\geq 0$, we have  (for $(w)_-=\max\{0, -w\}$)
\begin{align*}
W_i(t) & \geq  -C - \int_0^t  V_{Z_i}'(\sigma_i(s))\exp\left\{-\int_s^t V_{Z_i}'(\sigma_i(\tau))d\tau\right\}ds \|(W_{i+1})_-\|_{L^\infty(0,t)}\\
& \geq -C - \left(1- \exp\left\{-\int_0^t V_{Z_i}'(\sigma_i(\tau))d\tau\right\}\right) \|(W_{i+1})_-\|_{L^\infty(0,t)} \\
& \geq -C - \|(W_{i+1})_-\|_{L^\infty(0,t)}.
\end{align*}
We infer by induction that 
$$
\|(W_0)_-\|_{L^\infty(0,t)} \leq C i_0+ \|(W_{i_0})_-\|_{L^\infty(0,t)} = C i_0,
$$
since $W_{i_0}\geq  0$. This shows that 
$$
U_0(t)\geq \tilde U_0(t)- C{i_0}= \tilde U_0(0)+ \bar F(p-\ep)t-C i_0, 
$$
and therefore that 
$$
\liminf_{t\to+\infty} \frac{U_0(t)}{t} \geq  \bar F(p-\ep). 
$$
Since $i\to U_i(t)$ is nondecreasing, this holds for any $i\geq 0$. And, finally, as $\ep$ is arbitrary, this shows that 
$$
\liminf_{t\to+\infty} \frac{U_i(t)}{t} \geq  \bar F(p)\qquad \forall i\geq 0. 
$$
\medskip

It just remains to treat the case $p\in[0,\bar h_0]$. We begin to show that $\displaystyle{\lim_{p\to \bar h_0^+}\bar F(p)=0}$. Taking the limit $\theta\to 0^+$ in the equality $\E[V^{-1}_{Z_0}(\theta)]=\bar F^{-1}(\theta)$, we deduce that $\displaystyle{\lim_{\theta\to 0^+}\bar F^{-1}(\theta)=\bar h_0}$, which implies that 
$$\displaystyle{\lim_{p\to \bar h_0^+}\bar F(p)=0}.$$

 Now, let $p\in[0,\bar h_0]$. Using that $U_i(t)\ge 0$, we just have to show that $\limsup_{t\to +\infty} \frac {U_i(t)}t\le 0$.
 Let $\varepsilon>0$ and consider $\tilde U_i$  solution of 
$$
\frac{d}{dt} \tilde U_i(t)= V_{Z_i} (\tilde U_{i+1}(t)-\tilde U_i(t)), \; t\geq 0, \qquad \tilde U_i(0)= (\bar h_0+\varepsilon)i \qquad \forall i\geq 0. 
$$
Using the previous result (i.e., the case $p>\bar h_0$), we get that
$$\lim_{t\to +\infty}\frac {\tilde U_i(t)}{t}=\bar F(\bar h_0+\varepsilon).$$
Using that, by the comparison principle, we have $\tilde U_i(t)\ge  U_i(t)$ for all $i\ge 0$ and $t\ge 0$, we deduce that
$$\limsup_{t\to +\infty}\frac { U_i(t)}{t}\le\lim_{t\to +\infty}\frac {\tilde U_i(t)}{t}=\bar F(\bar h_0+\varepsilon).$$
Letting $\ep\to 0$ completes the proof of the proposition.
\end{proof}

\section{Proof of convergence}
We recall that the initial condition $(U^{\ep,0}_i)_{i\in \Z}$ is fixed and we consider the solution of
$$
\frac{d}{dt} U^\ep_i(t)= V_{Z_i}(U^\ep_{i+1}(t)-U^\ep_i(t)), \; t\geq 0, \; i\in \Z\qquad U^\ep_i(0)= U^{\ep,0}_i,\; i\in \Z. 
$$
We recall that we assumed the existence of a Lipschitz continuous function $u_0:\R\to \R$ such that 
$$
\lim_{\ep\to 0, \  \ep i \to x}  \ep U^{\ep,0}_i = u_0(x).
$$

\begin{proof}[Proof of Theorem \ref{th:CV}]
Let 
$$
u^*(x,t)= \limsup_{\ep\to 0, \ \ep (i,s)\to (x,t)} \ep U^\ep_i(s)\quad{\rm and}\quad u_*(x,t)= \liminf_{\ep\to 0, \ \ep (i,s)\to (x,t)} \ep U^\ep_i(s).
$$
We want to prove that $u^*$ and $u_*$ are respectively sub- and super-solution of \eqref{eq:macro}. Indeed, if we show these statements, then by the comparison principle for \eqref{eq:macro} (see Proposition \ref{pro:PCmacro}), we get $u^*\le u_*$. The reverse inequality being obvious, we get the convergence result.
\medskip

We begin with the initial condition. Since the velocity $V_Z$ is uniformly bounded and non-negative, we have that
$$\ep U_i^\ep (0)\le \ep U_i^\ep (s)\le \ep U_i^\ep(0)+ C\ep s.$$
Taking the $\limsup$ as $\ep\to 0$, with $\ep(i,s)\to (x,0)$, we infer that
$$u^*(x,0)=u_0(x).$$

We now turn to the equation. We only prove that $u^*$ is a sub-solution since the proof for $u_*$ is similar. Let $\phi$ be a smooth test function such that $u^*-\phi$ has a strict maximum at some point $(\bar x, \bar t)$. Then there exists a subsequence of $\ep$, still denoted in the same way,  and $(i_\ep, s_\ep)$, such that the map 
$$
(i,s)\to \ep U^\ep_i(s)- \phi(\ep i, \ep s)
$$
has a maximum at $(i_\ep, s_\ep)$. Moreover we have $(\ep i_\ep,\ep s_\ep)\to (\bar x,\bar t)$. The optimality of $(i_\ep, s_\ep)$ can be rewritten as 
\be\label{lajzrnzredHHH}
\frac{1}{\ep} \left( \phi(\ep(i_\ep +  i), \ep (s_\ep+ s))-  \phi(\ep i_\ep , \ep s_\ep)\right) \geq 
U^\ep_{i_\ep +i}(s)- U^\ep_{i_\ep}(s_\ep)
\qquad \forall (i,s). 
\ee
Let us fix $\delta_\ep>0$, with $\delta_\ep\to 0$ as $\ep\to 0$,  to be choose later.
Then by \eqref{lajzrnzredHHH}, we have, for any $i\geq 0$,  
\begin{align*}
U^\ep_{i_\ep +i}(s_\ep-\delta_\ep/\ep) & \leq  U^\ep_{i_\ep}(s_\ep) + \frac{1}{\ep} \left( \phi(\ep(i_\ep +  i), \ep s_\ep -\delta_\ep))-  \phi(\ep i_\ep , \ep s_\ep)\right) \\
& \leq  U^\ep_{i_\ep}(s_\ep) + \partial_x \phi(\ep i_\ep, \ep s_\ep-\delta_\ep)i + C\ep i^2 +
\frac{1}{\ep} \left( \phi(\ep i_\ep , \ep s_\ep -\delta_\ep))-  \phi(\ep i_\ep , \ep s_\ep)\right).
\end{align*}
Let us set $p= \partial_x\phi(\bar x,\bar t )$ and fix $\theta>0$ small. 
The RHS of the above inequality can be bounded from above as follows, for $\ep$ small enough: 
\begin{align}
U^\ep_{i_\ep +i}(s_\ep-\delta_\ep/\ep) & 
 \leq  U^\ep_{i_\ep}(s_\ep) + (p+\theta) i + 
\frac{1}{\ep} \left( \phi(\ep i_\ep , \ep s_\ep -\delta_\ep))-  \phi(\ep i_\ep , \ep s_\ep)\right)\notag \\ 
& \qquad \qquad \qquad \qquad \qquad  \forall i\in \{0, \dots, \theta/(C\ep)\} .\label{mjnzraerdHHH}
\end{align}

We now need the following localization lemma: 
\begin{Lemma}\label{lem.local}
Let $(U_i)$ and $(\tilde U_i)$ be two solutions to \eqref{eq:model} such that $U_i(0)\leq \tilde U_i(0)$ for $i\in \{0, \dots, K\}$ (where $K\geq 1$). Then 
$$
U_0(t)\leq \tilde U_0(t)+ V_{\max} t  \left(1-\exp(-\alpha t)\right)^K\qquad \forall t\geq 0, 
$$
where $\alpha=\sup_k \|\partial_x V_k\|_\infty$.
\end{Lemma}
We postpone the proof of the lemma and proceed with the ongoing argument. 
Let $(\tilde U_i)$ be the solution to 
$$
\frac{d}{dt} \tilde U_i(t)= V_{Z_i}(\tilde U_{i+1}(t)-\tilde U_i(t)), \; t\geq 0, \; i\in \N\qquad \tilde U_i(0)=  
(p+\theta) i,\; i\in \N. 
$$
We know from Proposition \ref{prop.limlin} that, a.s. in $\Omega_0$,  
$$
\lim_{t\to +\infty} \frac{1}{t}  \tilde U_0(t) = \bar F(p+\theta).
$$
We consider the flow generated by the left-hand side of \eqref{mjnzraerdHHH} (i.e., $(U^\ep_i)$) and  by its right-hand side (i.e., up to a constant, $(\tilde U_i)$) at time $\delta_\ep/\ep$. We have, by  Lemma \ref{lem.local} (with $t=\delta_\ep/\ep$ and $K= \theta/(C\ep)$): 
\begin{align}
U^\ep_{i_\ep}(s_\ep) & 
 \leq  U^\ep_{i_\ep}(s_\ep) +  \tilde U_0(\delta_\ep/\ep) + V_{\max} \frac {\delta_\ep}\ep \left(1-\exp(-\alpha \delta_\ep/\ep)\right)^{\theta/(C\ep)} \notag  \\
 & \qquad  +\frac{1}{\ep} \left( \phi(\ep i_\ep , \ep s_\ep -\delta_\ep))-  \phi(\ep i_\ep , \ep s_\ep)\right).\label{mjnzraerd}
\end{align}
Choosing $\delta_\ep:= -\gamma \ep \ln(\ep)$ where $\gamma>0$ is small enough (such that $\alpha\gamma<1$),  we have 
$$
\delta_\ep/\ep\to +\infty\qquad {\rm and}\qquad 
 \left(1-\exp(-\alpha \delta_\ep/\ep)\right)^{\theta/(C\ep)}\to 0.
$$ 
Dividing \eqref{mjnzraerd} by $\delta_\ep/\ep$ and letting $\ep\to 0$ gives
$$
0\leq \bar F(p+\theta) -\partial_t \phi(\bar x, \bar t).
$$
In view of the definition of $p$ and the fact that $\theta>0$ is arbitrary, we conclude that $u^\ep$ is a subsolution to the equation. 
\end{proof}

\begin{proof}[Proof of Lemma \ref{lem.local}]
Let $W_i(t):= \tilde U_i(t)-U_i(t)$. Then $(W_i)$ solves 
\begin{align*}
\frac{d}{dt} W_i(t) & = V_{Z_i}(\tilde U_{i+1}(t)- \tilde U_i(t))-  V_{Z_i}( U_{i+1}(t)-  U_i(t)) \\
& = V'_{Z_i}(\sigma_i(t)) (W_{i+1}(t)-W_i(t)),
\end{align*}
for some $\sigma_i$ between $ U_{i+1}(t)-  U_i(t)$ and $\tilde U_{i+1}(t)- \tilde U_i(t)$. 
We integrate this equation and find 
$$
W_i(t) = W_i(0) \exp\Bigl\{-\int_0^t V_{Z_i}'(\sigma_i(\tau))d\tau\Bigr\} + \int_0^t  V_{Z_i}'(\sigma_i(s))\exp\left\{-\int_s^t V_{Z_i}'(\sigma_i(\tau))d\tau\right\} W_{i+1}(s)ds.
$$
For $i\in \{0, \dots, K-1\}$, we have, since $W_i(0)\geq 0$ and $V_{Z_i}'\geq 0$, 
\begin{align*}
W_i(t)& \geq  - \| (W_{i+1})_-\|_{[0,t]}  \int_0^t  V_{Z_i}'(\sigma_i(s))\exp\{-\int_s^t V_{Z_i}'(\sigma_i(\tau))d\tau\} ds\\ 
& \geq - \| (W_{i+1})_-\|_{[0,t]}  \left( 1- \exp \{-\int_0^t V_{Z_i}'(\sigma_i(\tau))d\tau\}\right)  \\
& \geq  - \| (W_{i+1})_-\|_{[0,t]}  \left( 1- \exp \{-\alpha t\}\right)  .
\end{align*}
By induction, we infer that 
$$
\|(W_0)_-\|_{[0,t]} \leq \| (W_{K})_-\|_{[0,t]}  \left( 1- \exp \{-\alpha t\}\right)^K,
$$
where, because $V_Z$ is uniformly bounded by $V_{\max}$ and $W_K(0)\geq 0$, $\| (W_{K})_-\|_{[0,t]} \leq V_{\max} t$. 
\end{proof}
\vspace{10mm}
\noindent{\bf ACKNOWLEDGMENTS}
 
This project was co-financed by the European Union with the European regional development fund (ERDF,18P03390/18E01750/18P02733) and by the Normandie Regional Council via the M2SiNUM project and by ANR MFG (ANR-16-CE40-0015-01).

\bibliographystyle{siam}

\end{document}